\documentclass{amsart}

\usepackage[english]{my-shortcuts}
\usepackage{scalerel,stackengine}

\usepackage{a4wide}
\usepackage{amsmath,amsfonts,amssymb,latexsym,amsthm}
\usepackage{comment} 
\usepackage{graphicx}
\usepackage{caption} 

\newtheorem{theorem}{Theorem}

\newcommand{\EE}{\mathbb{E}}

\newcommand{\PP}{\mathbb{P}}
\newcommand{\RR}{\mathbb{R}}

\newcommand{\ve}{\varepsilon}

\stackMath
\newcommand\reallywidehat[1]{%
\savestack{\tmpbox}{\stretchto{%
  \scaleto{%
    \scalerel*[\widthof{\ensuremath{#1}}]{\kern-.6pt\bigwedge\kern-.6pt}%
    {\rule[-\textheight/2]{1ex}{\textheight}}
  }{\textheight}%
}{0.5ex}}%
\stackon[1pt]{#1}{\tmpbox}%
}
\begin{document}

\title{Polynomial approximations in a generalized Nyman-Beurling criterion}
\author{F. Alouges -- S. Darses -- E. Hillion}
\date{}


\address{Ecole polytechnique et cnrs, Institut Polytechnique de Paris, CMAP, Palaiseau, France} 
\email{francois.alouges@polytechnique.edu}

\address{Aix-Marseille Universit\'e, cnrs, Centrale Marseille, I2M, Marseille, France} 
\email{sebastien.darses@univ-amu.fr}

\address{Aix-Marseille Universit\'e, cnrs, Centrale Marseille, I2M, Marseille, France}
\email{erwan.hillion@univ-amu.fr}

\maketitle

\begin{abstract}

The Nyman-Beurling criterion, equivalent to the Riemann hypothesis (RH), is an approximation problem in the space of square integrable functions on $(0,\infty)$, involving dilations of the fractional part function by factors $\theta_k\in(0,1)$, $k\ge1$. Randomizing the $\theta_k$ generates new structures and criteria. One of them is a sufficient condition for RH that splits into (i) showing that the indicator function can be approximated by convolution with the fractional part, (ii) a control on the coefficients of the approximation. This self-contained paper generalizes conditions (i) and (ii) that involve a $\sg_0\in(1/2,1)$, and imply $\zeta(\sg+it)\neq 0$ in the strip $1/2<\sg\le\sg_0<1$. We then identify functions for which (i) holds unconditionally, by means of polynomial approximations. This yields in passing a short probabilistic proof of a known consequence of Wiener's Tauberian theorem. In this context, the difficulty for proving RH is then reallocated in (ii), which heavily relies on the corresponding Gram matrices, for which two remarkable structures are obtained. We show that 
a particular tuning
of the approximating sequence leads to a striking simplification of the second Gram matrix, then reading as a block Hankel form.
\end{abstract}

\begin{abstract}

Le crit\`ere de  Nyman-Beurling, \'equivalent \`a l'hypoth\`ese de Riemann (HR), est un probl\`eme d'approximation dans l'espace des fonctions de carr\'e int\'egrable sur $(0,\infty)$, par des dilatations de facteurs $\theta_k\in(0,1)$, $k\ge1$, de la fonction partie fractionnaire. Prendre les $\theta_k$ al\'etoires g\'en\`ere de nouvelles structures et de nouveaux crit\`eres.  L'un d'eux est une condition suffisante pour HR qui revient \`a (i) montrer que la fonction indicatrice peut \^etre approxim\'ee par des convolutions de la partie fractionnaire, et (ii) avoir un contr\^ole des coefficients de l'approximation. Ce papier g\'en\'eralise les conditions (i) et (ii) afin d'obtenir un crit\`ere impliquant $\zeta(\sg+it)\neq 0$ dans une bande  $1/2<\sg\le\sg_0<1$. On identifie ensuite des fonctions pour lesquelles (i) est v\'erifi\'ee inconditionnellement, gr\^ace \`a des approximations polynomiales. Cela fournit, au passage, une courte preuve probabiliste d'une cons\'equence connue d'un th\'eor\`eme Taub\'erien. Dans ce contexte, la difficult\'e à prouver HR se reporte sur (ii) qui pourrait n\'ecessiter une \'etude fine des matrices de Gram correspondantes. Nous obtenons deux structures remarquables de ces matrices. Nous montrons qu'un choix particulier des suites approximantes fournit une simplification frappante de la matrice de Gram qui s'\'ecrit alors sous la forme de matrices de Hankel par blocs.
\end{abstract}


\section{Introduction}

\subsection{Context and main definitions}

The Riemann hypothesis (RH) is equivalent to the Nyman-Beurling (NB) criterion, which is an approximation problem of the indicator function $\chi$ of $(0,1)$ in the space of square integrable functions on $(0,\infty)$, involving dilations of the fractional part function $\{\cdot\}$ by factors $\theta_k\in(0,1]$:
\begin{theorem}[\cite{BDBLS00}] \label{th:NBcriterion}
RH holds if and only if,
given $\ve>0$, there exist $n \ge 1$, coefficients $c_1,\ldots,c_n \in \RR$, and $\theta_1,\ldots,\theta_n \in (0,1]$ such that 
\begin{equation} 
\int_0^\infty \left( \chi(t) - \sum_{k=1}^n c_k \left\{ \frac{\theta_k}{t} \right\}\right)^2 dt < \ve.
\label{eq:NBexplicit}
\end{equation}
\end{theorem}
B\'aez-Duarte \cite{BD03} showed that it is possible to specify $\theta_k = \frac1k$ in this criterion. The coefficients $(c_k)_{1\leq k\leq n}$ leading to the best approximation, then solve the linear system
$$
Gx=b,
$$
where $x=(c_1,\cdots,c_n)^T,\, b=(b_1,\cdots,b_n)^T$, with $\d b_k =\int_0^1 \left\{ \frac{1}{kt}\right\}dt$ and $G$ is the Gram matrix
$$
G_{k,l} = \int_0^{\infty} \left\{ \frac{1}{kt}\right\}\left\{ \frac{1}{lt}\right\} \ dt,\,\,\mbox{ for }1\le k,l\le n\,.
$$
The computation of $G$ is possible through a formula due to Vasyunin \cite{Vas95}. The evaluation of the distance in (\ref{eq:NBexplicit}) may be then approximately evaluated by numerical means \cite{BDBLS00}.

Over recent years, some research works have been devoted to the Nyman-Beurling criterion for the Riemann Hypothesis. For example, the articles \cite{BC13} by Bettin and Conrey, or \cite{MR15}, \cite{MR16} by Maier and Rassias aim to study cotangent sums related to this problem and present a series of relevant results.



More recently, randomizing the $\theta_k$ has produced new characterizations and structures \cite{DH21}. Among these, the following sufficient condition for RH is obtained.
From now on, we write $\pE Z$ for the expectation of a random variable (r.v.) $Z$.

\begin{theorem} \label{th:gNBtoRH}
Let $(Z_{k,n})_{1\le k\le n, n \ge 1}$ be positive square integrable r.v. satisfying, for any $\alpha >1$,
\bea \label{moment}
\sum_{k=1}^n \left(\pE Z_{k,n}^2\right)^{\alpha/2} \ll_\alpha 1.
\eea
If there exist coefficients $(c_{k,n})_{1 \le k \le n, n \ge 1}$ such that, for any $M_n\to\infty$, 
\bea\label{Dn2}
D_n^2 = \int_0^\infty \left|\chi(t)-\sum_{k=1}^n c_{k,n}\pE\left\{\frac{Z_{k,n}}{t}\right\} \right|^2 dt & \xrightarrow[n \rightarrow \infty]{} & 0 \\
\label{eq:cknBound}
\sum_{k=1}^n |c_{k,n}|^2\PP(Z_{k,n} \ge M_n) & \xrightarrow[n\to\infty]{} & 0,
\eea
then RH holds.
\end{theorem}
The relevance of the former theorem stems from the fact that it is possible to show that the converse holds for some specific structures, as dilated or concentrated r.v., see \cite{DH21} for more details. A particular case of dilated structure is given by exponentially distributed r.v. $Z_k\sim\cal E(k)$, which is related to a particular cotangent sum, studied in~\cite{DH20}.

Notice that, for r.v. $Z_{k}$ with densities $\phi_{k}$, the involved functions
\begin{equation}
    \pE\left\{\frac{Z_{k}}{t}\right\}  =  \int_0^\infty \left\{\frac{x}{t}\right\}\phi_{k}(x)dx,
\end{equation}
can be written as a multiplicative convolution. Indeed, for $g:\R_+\to\R$, let $g^\times(t)$ denote the multiplicative convolution of $g$ with the fractional part:
\begin{equation}
g^\times(t) = \int_0^\infty \left\{\frac{x}{t}\right\}g(x)\frac{dx}{x}=\left(\left\{\frac{1}{\cdot}\right\}*g\right)(t).
\end{equation}
If $g_k(x)=x\phi_k(x)$ then
$$
g^\times_k(t)=\pE\left\{\frac{Z_{k}}{t}\right\}.
$$
We then observe that this latter expression allows for a generalization with functions $g_k$ that possibly change sign. 

From now on, we use two different notations: $g^\times_k$ when dealing with functions $g_k$, and $h^\times_k$ when dealing with r.v. $Z_k$:
\begin{eqnarray}
g_k^\times(t) & := & \int_0^\infty \left\{\frac{x}{t}\right\}g_k(x)\frac{dx}{x} \\
h^\times_k(t) & := & \pE\left\{\frac{Z_{k}}{t}\right\}.
\end{eqnarray}

\begin{defi}
We say that a sequence of real function $(g_k)_k$, resp. a sequence of r.v. $(Z_k)_k$, verifies {\emph gNB} if there exist coefficients $(c_{k,n})_{1 \le k \le n, n \ge 1}$ such that 
\bea
\d D_n^2 = \int_0^\infty \left|\chi(t)-\sum_{k=1}^n c_{k,n} g^\times_{k}(t) \right|^2 dt \xrightarrow[n \rightarrow \infty]{} 0, \qquad ({\rm gNB})
\eea
resp. if 
\bea
\d D_n^2 = \int_0^\infty \left|\chi(t)-\sum_{k=1}^n c_{k,n} h^\times_{k}(t) \right|^2 dt \xrightarrow[n \rightarrow \infty]{} 0. \qquad ({\rm gNB})
\eea
\end{defi}
The main purpose of this paper is to identify classes of functions $g_k$, or corresponding r.v. $Z_k$, that verify gNB, unconditionally (i.e. without assuming RH):
\begin{enumerate}
    \item In Section \ref{ex1}, we treat the case of r.v. $\d Z_k = Y/X_k$ where $Y$ is a positive r.v. and $X_k$ are $\Gamma(k,1)$-distributed r.v. independent of $Y$. 
    \item In Section \ref{ex2}, we consider functions $g_k$, that are not necessarily non negative, defined by induction as, for all $k\ge0$:
\bean
g_{k+1}(x) = - x g'_k(x) - r_k g_k(x), \quad x\ge0,
\eean
where $(r_k)_{k\ge0}$ is a real sequence and $g_0$ is a suitable initialization.
\end{enumerate}
In each case, we  also provide expressions of the scalar products $\langle h^\times_{k}, h^\times_{j} \rangle$, and $\langle g^\times_{k}, g^\times_{j} \rangle$, which may be necessary to tackle Condition (\ref{eq:cknBound}) in the future. Two remarkable structures are obtained for the Gram matrices. Moreover, taking $r_k=1/2$ in (2) yields to a striking simplification.

To handle Case (2), we need to generalize Theorem \ref{th:gNBtoRH}, which is done in Section \ref{basic}.


\subsection{Preliminaries}

We say that $g:\R^+\to\R$ verifies Assumption (M) if $$\int_0^\infty|g(x)|dx<\infty \quad {\rm and}\quad \int_0^\infty|g(x)|\frac{dx}{x}<\infty .\qquad (\rm M)$$
We write a complex number $s=\sg+it$, $\sg,t\in\R$. The previous assumption allows to define the Mellin transform $\widehat{g}$ of $g$ in the critical strip $0<\sg<1$:
\bean
\widehat{g}(s) & = & \int_0^\infty g(x)x^{s-1}dx,
\eean
since $\d \int_0^\infty |g(x)|x^{\sg-1}dx\le 
\left(\int_0^\infty|g(x)|dx\right)^{\sg}\left(\int_0^\infty|g(x)|\frac{dx}{x}\right)^{1-\sg}$, due to the H\"older inequality. Assumption (M) is sufficient but not necessary for the Mellin transform to be defined, and we recall that Mellin-Plancherel theory allows to define $\widehat{g}(s)$ whenever $g\in L^2(0,\infty)$. In this case, we have the isometry:
\bean
\int_0^\infty |g(x)|^2dx & =  & \frac{1}{2\pi}\int_{-\infty}^{\infty} \left|\widehat{g}\left(\frac{1}{2}+it\right)\right|^2dt.
\eean
We finally recall the fundamental identity on which relies the NB criterion, see \cite[(2.1.5)]{Tit86}:
\bea \label{fonda}
\int_0^\infty \left\{\frac{1}{x}\right\}x^{s-1}dx & = & -\frac{\zeta(s)}{s}, \quad 0<\sg<1.
\eea

In the case where $\phi(x) = g(x)/x$ is the density of a r.v. $Z\ge0$, Assumption (M) simply translates into $\pE Z<\infty$. 
The following lemma is standard, but we give a proof in our framework for the sake of completeness.

\begin{lemm}\label{M2}
{\rm (i)} Let $g:\R^+\to\R$ satisfying (M).
Then $g^\times\in L^2(\R_+)$ and its Mellin transform $\widehat{g^\times}(s)$ is well defined for $\sg \in (0,1)$. Moreover, 
\bean
\widehat{g^\times}(s) & = & -\frac{\zeta(s)}{s}\widehat{g}(s),\quad 0<\sg<1\,.
\eean
{\rm (ii)} If $Z\ge0$ is an integrable r.v. then $t\longmapsto h^\times(t): =\pE\left\{\frac{Z}{t}\right\}$ belongs to $L^2(\R_+)$ and 
\bean
\widehat{h^\times}(s) & = & -\frac{\zeta(s)}{s} \pE Z^s,\quad 0<\sg<1\,.
\eean
\end{lemm}

\begin{proof}
(i) Writing $|g(x)|=|g(x)|^{1/2}|g(x)|^{1/2}$ and using the Cauchy-Schwarz inequality, one obtains:
\bean
\int_0^\infty |g^\times(t)|^2dt & \le & \int_0^\infty \left(\int_0^\infty\left\{\frac{x}{t}\right\}|g(x)|\frac{dx}{x}\right)^2dt \\    
& \le & \int_0^\infty \left(\int_0^\infty|g(x)|\frac{dx}{x}\right)\left(\int_0^\infty\left\{\frac{x}{t}\right\}^2|g(x)|\frac{dx}{x}\right) dt\\
    &  = & \int_0^\infty|g(x)|\frac{dx}{x} \int_0^\infty|g(x)| dx
\int_0^\infty\left\{\frac{1}{t}\right\}^2 dt <\infty,
\eean
due to Assumption (M). We already noticed that 
for all $\sg\in(0,1)$, $\d \int_0^\infty|g(x)|x^{\sg-1}dx <+\infty$, so 
\begin{equation*}
    \int_0^\infty \int_0^\infty \left|\left\{\frac{x}{t}\right\}\frac{g(x)}{x} t^{s-1}\right|dxdt \le \int_0^\infty  \left\{\frac{1}{t}\right\}t^{\sg-1}dt \int_0^\infty|g(x)|x^{\sg-1}dx <+\infty\,.
\end{equation*}
Hence $\widehat{g^\times}$ is well defined and 
we can apply Fubini's theorem, which justifies:
$$
\widehat{g^\times}(s) = \reallywidehat{\displaystyle\left(\left\{\frac{1}{\cdot}\right\}*g\right)}(s) = \widehat{\left\{\frac{1}{\cdot}\right\}}(s) \widehat{g}(s) = -\frac{\zeta(s)}{s}\widehat{g}(s).
$$
(ii) Now, let $Z\ge0$ be a integrable r.v. and set $h^\times(t) =\pE\left\{\frac{Z}{t}\right\}$. The situation reads even simpler:
\bean
\int_0^\infty |h^\times(t)|^2dt \le \pE \int_0^\infty \left\{\frac{Z}{t}\right\}^2 dt = \pE Z \ \int_0^\infty \left\{\frac{1}{t}\right\}^2 dt <\infty\,.
\eean
Again, $\widehat{h^\times}$ is well defined by Fubini, and the result simply follows from a change of variable:
\bean 
\widehat{h^\times}(s) = \int_0^\infty \pE\left\{\frac{Z}{t}\right\} t^{s-1} dt = \pE \int_0^\infty \left\{\frac{Z}{t}\right\} t^{s-1} dt = \pE \left[ Z^s \int_0^\infty \left\{\frac{1}{u}\right\} u^{s-1} du \right] = -\frac{\zeta(s)}{s} \pE Z^s,
\eean 
without even assuming that $Z$ has a density.
\end{proof}

\begin{rem}\label{rem1}
Notice that weaker assumptions, such as $\pE Z^\sg<\infty$, $0<\sg<1$, are sufficient to show $h^\times\in L^2$, following the first lines of the proof with a different use of the Cauchy-Schwarz inequality. 
\end{rem}

Since the methodology developed hereafter mainly uses the density of polynomials in a suitable weighted $L^2$ space, we recall the classical following result:
\begin{lemm}[\cite{Nik12}, (c) p.76] \label{densite}
Let $\nu$ be a positive measure on $\R$ with $\d \int_\R e^{a|t|}d\nu(t)<\infty$ for some $a>0$. Then the polynomials are dense into $L^2(\R,\nu)$.
\end{lemm}

\section{A sufficient condition for $\zeta(s)\neq 0$ in the strip $\frac12<\sg\le\sg_0<1$} \label{basic}

We generalize Theorem \ref{th:gNBtoRH} replacing $\pE\{Z_k/t\}$ by general functions $g_k^\times$, which allows us to obtain some sharper estimates.
\begin{theorem}
Let $\sg_0\in(\frac12,1)$. Let $(g_{k})_{k\ge 1}$ be functions verifying Assumption (M). \\
Let $\phi:\R_+\to\R$ be such that $\d \int_0^{\infty}(1+t)\left|\phi(t) \right|^2dt<\infty$ and $\widehat{\phi}$ does not vanish in $\frac12<\sg\le\sg_0$.

If there exist real coefficients $(c_{k,n})_{1 \le k \le n}$ and $M_n\to\infty$ such that
\bea 
M_n^{2\sg_0-1}\int_0^\infty \left(\phi(t)-\sum_{k=1}^n c_{k,n}g_k^\times(t) \right)^2 dt \xrightarrow[n \rightarrow \infty]{} 0 \label{Dg}\\
M_n^{\sigma_0-1} \int_{M_n}^\infty \left|\sum_{k=1}^n c_{k,n} g_k(t)\right|dt  \xrightarrow[n\to\infty]{} 0, \label{Cg}
\eea
then $\zeta$ does not vanish in the strip $\frac12<\sg\le\sg_0$.
\end{theorem}

\begin{proof}
Let us assume that we are given $M_n\ge1$ tending to infinity and real coefficients $(c_{k,n})_{1 \le k \le n}$ such that (\ref{Dg}) and (\ref{Cg}) hold. Notice that $\widehat{\phi}$ exists in the strip $\frac12<\sg<1$ since
\bean
\left(\int_0^1\left|\phi(t) \right|t^{\sg-1}\,dt \right)^2 &\le & \frac{1}{2\sg-1}\int_0^1 \left|\phi(t) \right|^2dt <\infty \\
\left(\int_{1}^{\infty}\left|\phi(t) \right|t^{\sg-1}\,dt \right)^2& \le & \int_{1}^{\infty}t\left|\phi(t) \right|^2dt \int_{1}^{\infty}t^{2(\sg-1-\frac12)}\,dt<\infty,
\eean
due to $2\sg-1>0$ and $3-2\sg>1$. Set
\bean
\epsilon_n^2 & = & \int_{M_n}^{\infty}t\left|\phi(t) \right|^2dt.
\eean
Notice that $\e_n\xrightarrow[n\to\infty]{}0$ and, again by the Cauchy-Schwarz inequality, 
\bea
\int_{M_n}^{\infty}\left|\phi(t) \right|t^{\sg-1}\,dt = \int_{M_n}^{\infty}t^{1/2}\left|\phi(t) \right|t^{\sg-3/2}\,dt \le \frac{\e_n M_n^{\sg-1}}{(2-2\sigma)^\frac12}. \label{queueMellin} 
\eea
We write for simplicity $\d G_n=\sum_{k=1}^n c_{k,n}g_k$, so
\bean
G_n^\times=\left\{\frac{1}{\cdot}\right\}*G_n=\sum_{k=1}^n c_{k,n}g_k^\times.
\eean

The first lines of the proof follows the same ones as those of the Nyman-Beurling criterion.\\ Assume for contradiction that $\zeta(s)=0$, for some (now fixed) $s=\sg+i\tau\in\C$ with $1/2<\sg\le\sg_0$.\\ 
Using $\d \widehat{G_n^\times}(s)  =  -\frac{\zeta(s)}{s}\widehat{G_n}(s)=0$, we then obtain for all integer $n$:
\bean
\widehat{\phi}(s) & = & \int_0^{\infty}\left(\phi(t) -G_n^\times(t)\right)t^{s-1}\,dt.
\eean
We will show that the right hand side tends to $0$ as $n\to\infty$, contradicting $\widehat{\phi}(s)\neq0$.\\ Following the proof in \cite{DH21}, we write:
\bean
\left|\int_0^{\infty}\left(\phi(t) -G_n^\times(t)\right)t^{s-1}dt\right| \le  \left|\int_0^{M_n}\left(\phi(t) -G_n^\times(t)\right)t^{s-1}dt\right| + \left|\int_{M_n}^{\infty}\left(\phi(t) -G_n^\times(t)\right)t^{s-1}dt\right| = I + II\,.
\eean

We start to estimate $I$, by the Cauchy-Schwarz inequality:
\bean
I &\ll & M_n^{\sg-1/2}\left(\int_0^{M_n}\left(\phi(t) -G_n^\times(t)\right)^2dt\right)^{1/2}\,.
\eean

We now turn to $II$. Using the triangular inequality and $G_n^\times=\left\{\frac{1}{\cdot}\right\}*G_n$,
\bean
II & \le &  \int_{M_n}^{\infty}\left|\phi(t) \right|t^{\sg-1}dt + \int_{M_n}^{\infty} \left|\int_0^{M_n}\left\{\frac{x}{t}\right\}\frac{G_n(x)}{x}\,dx\,\right|t^{\sg-1}dt + \int_{M_n}^{\infty} \left|\int_{M_n}^{\infty}\left\{\frac{x}{t}\right\}\frac{G_n(x)}{x}\,dx\right|t^{\sg-1}dt.
\eean
Noticing that $\d \left\{\frac{x}{t}\right\}=\frac{x}{t}$ if $x\le M_n \le t$, and $\displaystyle 0\le \left\{\frac{x}{t}\right\}\frac{t}{x} \le 1$, we obtain
\bean
\int_{M_n}^{\infty} \left|\int_0^{M_n}\left\{\frac{x}{t}\right\}\frac{G_n(x)}{x}\,dx\,\right|t^{\sg-1}dt & = & \frac{M_n^{\sg -1}}{1-\sg}\left|\int_0^{M_n}G_n(x)\,dx\right| \\
\int_{M_n}^{\infty} \left|\int_{M_n}^{\infty}\left\{\frac{x}{t}\right\}\frac{G_n(x)}{x}\,dx\right|t^{\sg-1}dt  & \le & \frac{M_n^{\sg -1}}{1-\sg}\int_{M_n}^{+\infty}\left|G_n(x)\right|\,dx. 
\eean
Hence, with (\ref{queueMellin}),
\bea \label{boundII}
 II   &\le & \frac{\e_n M_n^{\sg-1}}{(2-2\sigma)^\frac12} +\frac{M_n^{\sg -1}}{1-\sg}\left|\int_0^{M_n}G_n(x)\,dx\right| + \frac{M_n^{\sg -1}}{1-\sg}\int_{M_n}^{\infty}\left|G_n(x)\right|\,dx.
\eea
The first and third term of the right hand side tend to zero as $n\to\infty$ by hypothesis. We can estimate $ \int_0^{M_n}G_n(x)\,dx$ as in \cite{DH21}, in the spirit of "BDBLS's trick" (see \cite[Lemme 1 \& Prop. 1 p.133]{BDBLS00}), which removes the boundary condition in the 1950' NB criterion. \\
On one hand, we first notice that
\bea
\left(\int_0^{M_n}G_n(x)\,dx\right)^2 & = & M_n \int_{M_n}^{+\infty} \frac{dt}{t^2}\left(\int_0^{M_n}G_n(x)\,dx\right)^2 \nonumber\\
    & = & M_n \int_{M_n}^{+\infty} \left|\int_0^{M_n}\frac{x}{t}\frac{G_n(x)}{x}\,dx\right|^2\,dt \nonumber\\
    & = & M_n \int_{M_n}^{+\infty} \left|\int_0^{M_n}\left\{\frac{x}{t}\right\}\frac{G_n(x)}{x}\,dx\right|^2\,dt, \label{eq:Mn}
\eea
still using $\d \left\{\frac{x}{t}\right\}=\frac{x}{t}$ if $x\le t$.\\
On the other hand, writing $$ \int_0^{M_n}\left\{\frac{x}{t}\right\}\frac{G_n(x)}{x}dx=\int_0^{\infty}\left\{\frac{x}{t}\right\}\frac{G_n(x)}{x}\,dx-\phi(t)-\int_{M_n}^\infty\left\{\frac{x}{t}\right\}\frac{G_n(x)}{x}\,dx+\phi(t),$$ using $(a+b+c)^2 \le 3(a^2+b^2+c^2)$ and again $\d 0\le \left\{\frac{x}{t}\right\}\frac{t}{x} \le 1$, we have
\bean
\int_{M_n}^{\infty} \left|\int_0^{M_n}\left\{\frac{x}{t}\right\}\frac{G_n(x)}{x}\,dx\right|^2\,dt &\le &
3\int_{M_n}^{\infty} \left( \int_0^{\infty}\left\{\frac{x}{t}\right\}\frac{G_n(x)}{x}\,dx - \phi(t)\right)^2\,dt  \\
&&+ 3 \int_{M_n}^{\infty} \left|\int_{M_n}^{\infty}\left\{\frac{x}{t}\right\}\frac{G_n(x)}{x}\,dx\right|^2 dt+ 3\int_{M_n}^{\infty}\phi(t)^2dt\\
&\le &
3\int_{M_n}^{\infty} \left(\phi(t) -G_n^\times(t)\right)^2\,dt \\
&&+ 3 M_n^{-1} \left(\int_{M_n}^{\infty}|G_n(x)|\,dx\right)^2+ 3\int_{M_n}^{\infty}\phi(t)^2dt\,.
\eean
Due to (\ref{eq:Mn}) and the definition of $\epsilon_n$, this leads to 
\bean
\left(\int_0^{M_n}G_n(x)\,dx\right)^2 &\le & 3M_n\int_{M_n}^{\infty} \left(\phi(t)-G_n^\times(t)\right)^2\,dt + 3\left(\int_{M_n}^{\infty}|G_n(x)|\,dx\right)^2 +3\ \e_n^2,
\eean
which, put back in the bound (\ref{boundII}) of II with $\sqrt{a+b+c}\le \sqrt{a}+\sqrt{b}+\sqrt{c}$, gives
\bean
II &\ll& M_n^{\sg -\frac12}\left(\int_{M_n}^{\infty} \left(\phi(t)-G_n^\times(t) \right)^2\,dt\right)^{1/2} + M_n^{\sg -1}\int_{M_n}^{\infty}\left|G_n(x)\right|\,dx \ +\ M_n^{\sg-1}\e_n.
\eean
Combining the bounds for I and II, we thus end up with
\bean
\left|\int_0^{\infty}\left(\phi(t) -G_n^\times(t)\right)t^{s-1}\,dt\right| &\ll &  M_n^{\sg -\frac12}\left\|\phi -G_n^\times\right\|_{L^2} + M_n^{\sg -1}\int_{M_n}^{\infty}\left|G_n(x)\right|\,dx +\ \e_n M_n^{\sg-1}.
\eean
We eventually notice that the three last terms of the right hand side tend to $0$ as $n\to\infty$ due to $M_n^\sg\le M_n^{\sg_0}$ (since $M_n\ge1$ and $\sg\le\sg_0<1$), (\ref{Dg}) and (\ref{Cg}).
\end{proof}

\section{Class of Inverse Gamma distributions} \label{ex1}

In this section, we deal with the probabilistic framework of Lemma \ref{M2} (ii), and consider a sequence of r.v. $\d Z_k = Y/X_k$ where $Y\ge0$ is an integrable r.v. and $X_k$ is a $\Gamma(k,1)$-distributed r.v. independent of $Y$. Hence, in this probabilistic context, we set
\bean
h_k^\times(t) & = & \pE\left\{\frac{Y}{X_k t}\right\}.
\eean

\subsection{Inverse Gamma distributions verify gNB}
\begin{theorem} \label{gNB1}
The sequence $(Z_k)_{k \ge 1}$ verifies gNB.
\end{theorem}

\begin{proof}
We recall that the density of the $\Gamma(k,1)$ distribution is $f_k(x) = \frac{1}{\Gamma(k)} x^{k-1} e^{-x}$, $k\ge1$. We have for all $k\ge2$, $\pE Z_k=\pE Y\ \pE (1/X_k) <\infty$. For $k=1$, we only have $\pE (1/X_1)^\sg<\infty$ for all $\sg\in(0,1)$, which by Remark \ref{rem1} is also sufficient to ensure that $g_1^\times \in L^2(\R_+)$. Therefore we can apply Lemma \ref{M2} (ii): for all $k\ge1$,
\begin{eqnarray*}
\widehat{h_k^\times}(s) & = & -\frac{\zeta(s)}{s}\pE Z_k^s, \quad 0<\sg<1.
\end{eqnarray*}
Since $X_k$ and $Y$ are independent,
\begin{eqnarray*}
\pE Z_k^s = \pE Y^s\  \pE X_k^{-s}, \quad 0<\sg<1.
\end{eqnarray*}
Moreover, we can compute
\begin{eqnarray*}
\pE X_k^{-s}= \frac{1}{\Gamma(k)} \int_0^\infty x^{k-s-1} e^{-x}dx =\frac{\Gamma(k-s)}{\Gamma(k)}.
\end{eqnarray*}
Notice that, for $k\ge2$ and $0<\sg<1$, 
\begin{eqnarray*}
\Gamma(k-s) & = & (k-1-s)\Gamma(k-1-s)\\
    & = & (k-1-s)\cdots (1-s)\Gamma(1-s)\\
    & = & P_{k-1}(s)\Gamma(1-s),
\end{eqnarray*}
where $P_0=1$ and $P_{k}(s)=(k-s)(k-1-s)\cdots(1-s)$ is a  polynomial of degree $k$, known as a Pochhammer symbol.

We then write for some coefficients $c_{k,n}$, setting $s=1/2+it$,
\bean
D_n^2 = \int_0^\infty \left|\chi(x)-\sum_{k=1}^n c_{k,n} h^\times_{k}(x) \right|^2 dx 
    & = & \frac{1}{2\pi}\int_{-\infty}^\infty \left|\widehat{\chi}(s)-\sum_{k=1}^n c_{k,n} \widehat{h^\times_{k}}(s) \right|^2 dt \\
    & = & \frac{1}{2\pi}\int_{-\infty}^\infty \left|\frac{1}{s}+\sum_{k=1}^n c_{k,n} \frac{\zeta(s)}{s}\pE Y^s \Gamma(1-s)\frac{P_{k-1}(s)}{\Gamma(k)}\right|^2 dt \\
    & = & \frac{1}{2\pi}\int_{-\infty}^\infty \left|\frac{1}{s\varphi(s)}+\sum_{k=1}^n \frac{c_{k,n}}{\Gamma(k)} P_{k-1}(s)\right|^2 |\varphi(s)|^2 dt\,,
\eean
where $\d \varphi(s)=\frac{\zeta(s)}{s}\pE Y^s \Gamma(1-s)$. 

We now notice that $\zeta(s)=O(t)$ (see e.g. \cite[Corollary 3.7 p.234]{Ten95}), $|\pE Y^s|\le \pE \sqrt{Y}$, and by the complement formula 
\bean
|\Gamma(1-s)|^2=\Gamma\left(\frac{1}{2}-it\right)\Gamma\left(\frac{1}{2}+it\right)=\frac{\pi}{\sin\left(\frac{\pi}{2}+i\pi t\right)} \ll e^{-\pi t}.
\eean
Hence, we can apply Lemma \ref{densite} with the measure $\nu(dt)=|\varphi(s)|^2 dt$ and the function $t\mapsto 1/(s\varphi(s))$ that belongs to $L^2(\R,\nu)$. Therefore, since the polynomials $P_k$ are graduated, there exist coefficients $c_{k,n}$ such that $D_n\to 0$. In other words, the family $(Y/X_k)_{k\ge1}$ verifies gNB.
\end{proof}

As a consequence of Wiener's Tauberian theorem (see e.g. \cite[Theorem 2 p.25]{Bal00}), the following result is known:
\begin{theorem}
Given $\ve>0$, there exist $m \ge 1$, $c'_1,\ldots,c'_m \in \RR$, and $\theta_1>0,\ldots,\theta_m>0$ such that 
\begin{equation} \label{}
\int_0^\infty \left( \chi(t) - \sum_{l=1}^m c'_l \left\{ \frac{\theta_l}{t} \right\}\right)^2 dt < \ve.
\end{equation}
\end{theorem}
Interestingly, the existence of r.v. verifying gNB in Theorem \ref{gNB1} provides a short probabilistic proof of this fact. The basic idea is to approximate $\pE\{Z_k/t\}$ in $D_n^2$ by $\frac{1}{N}\sum_{j=1}^N\{Z_{k,j}/t\}$ where $(Z_{k,j})_j$ are independent copies of $Z_k$ defined on a probability space $(\Omega,\cal F,\bP)$. Then we use Erd\H{o}s probabilistic method through the following instance: if $\pE V\le \e$ for a r.v. $V\ge0$, then there exists $\omega\in \Omega$, s.t. $V(\omega)\le\e$. 
Notice that $Z_k=X/Y_k$ are fully supported on $\R_+$, and this is why the $\theta$'s do not lie in $(0,1)$.
\begin{proof}
Fix $\e>0$. We know by Theorem \ref{gNB1} that there exist $n$ and coefficients $c_1,\ldots,c_n \in \RR$ (that are fixed now) such that 
\bean
D_n^2 = \int_0^\infty \left(\chi(t)-\sum_{k=1}^n c_{k}\pE\left\{\frac{Z_{k}}{t}\right\} \right)^2 dt<\e\,.
\eean
Then set $\d d^2_{n,N} = \int_0^\infty \left(\chi(t)-\sum_{k=1}^n c_{k}/N\sum_{j=1}^N\{Z_{k,j}/t\} \right)^2 dt$, and write, using $(a+b)^2\le 2a^2+2b^2$,
\bean
\pE d^2_{n,N} 
    &\le & 2D_n^2 +
2\pE\int_0^\infty \left(\sum_{k=1}^n c_{k}\left(\pE\{Z_k/t\}-\frac{1}{N}\sum_{j=1}^N\{Z_{k,j}/t\}\right) \right)^2 dt = 2D_n^2 +2R_{n,N}^2\,.
\eean
Using now the Cauchy-Schwarz inequality,
\bean
R_{n,N}^2 \le \sum_{k=1}^n c_{k}^2 \sum_{k=1}^n\int_0^\infty \pE\left(\pE\{Z_k/t\}-\frac{1}{N}\sum_{j=1}^N\{Z_{k,j}/t\} \right)^2 dt
\le \frac{1}{N} \sum_{k=1}^n c_{k}^2 \sum_{k=1}^n\int_0^\infty {\rm Var}(\{Z_k/t\}) dt.
\eean
We can then choose $N$ sufficiently large so that $R_{n,N}^2 \le \e$. Hence 
$\pE d^2_{n,N}\le 4\e$, so there exists $\omega\in\Omega$ such that $d^2_{n,N}(\omega)\le 4\e$, which concludes the proof, the desired $\theta_l$'s being the $Z_{k,j}(\omega)$.
\end{proof}

\subsection{Remark on a specific distribution tail}
The preceding arguments generalize a remark due to Vincent Alouin who noticed that the distribution tail 
$x\mapsto (1+x)^{-k}$ has a Mellin transform which satisfies, by the change of variables $u=1/(1+x)$, i.e. $x=1/u-1$, the identity
\beq \label{alouin}
\int_0^\infty \frac{x^{s-1}}{(1+x)^k} dx = - \int_0^1 \left(\frac{1}{u}-1\right)^{s-1} u^k \frac{du}{-u^2} = \int_0^1 (1-u)^{s-1} u^{k-s-1} du = \frac{\Gamma(s) \Gamma(k-s)}{\Gamma(k)},
\eeq
which is a product of a polynomial in $s$ by a fixed (independent of $k$) function. It turns out that the corresponding r.v. are of the form $Z_k\sim Y/X_k$ where $Y\sim \cal E(1)$.

\subsection{Computation of the Gram matrix}
We now want to compute the corresponding scalar products. We introduce the functions 
\bean 
\rho(t) & = & \pE\left\{\frac{Y}{t}\right\},\quad  t>0, \\
A(u) & = & \int_0^\infty \rho(ut) \rho((1-u)t) dt, \quad 0<u<1.
\eean 
Two interesting particular cases are:
\begin{itemize}
    \item[(i)] when $Y \sim \delta_1$, we have $\rho(t) = \left\{\frac{1}{t}\right\}$\,,
    \item[(ii)] when $Y \sim \mathcal{E}(\lambda)$, we have $\d \rho(t) = \frac{1}{e^{\lambda t}-1}-\frac{1}{\lambda t}$\,.
\end{itemize}

Recall that $
h_k^\times(t) = \pE\left\{\frac{Y}{X_k t}\right\}$ where $Y$ is a r.v. satisfying (M) independent of $X_k\sim \Gamma(k,1)$. 

\begin{prop}
For $m,n \ge 0$, we have
\bean
\left\langle h_{n+1}^\times,h_{m+1}^\times \right\rangle & = & \int_0^1 B_{n}^{m+n}(u) A(u)du, 
\eean
where $\d B_{n}^{m+n}(u)={n+m\choose n}u^n(1-u)^m$ is an elementary Bernstein polynomial. 
\end{prop}

\begin{proof} We first notice that, since $X_n$ and $Y$ are independent, 
\bean
h_{n}^\times(t) = \pE\left\{\frac{Y}{X_n t}\right\} &=& \EE\left[\rho(X_n t) \right]\\
&=& \frac{1}{(n-1)!} \int_0^\infty \rho(xt) x^{n-1} e^{-x} dx \\
&=& \frac{1}{(n-1)! t^n} \int_0^\infty \rho(x) x^{n-1} e^{-\frac{x}{t}} dx.
\eean 

We then have, by Fubini,
\bean
\left\langle h_{n+1}^\times,h_{m+1}^\times \right\rangle &=& \frac{1}{m!n!} \int_0^\infty \int_0^\infty \rho(x) \rho(y) x^n y^m \int_0^\infty \frac{1}{t^{m+n+2}} e^{-\frac{x+y}{t}} dt\,dx\,dy \\
&=& \frac{1}{m!n!} \int_0^\infty \int_0^\infty \rho(x) \rho(y) x^n y^m \frac{(m+n)!}{(x+y)^{n+m-1}} dx\,dy\,,
\eean
where we used the elementary formula $\d \int_0^\infty \frac{1}{t^\alpha} e^{-\frac{\beta}{t}} dt = \frac{\Gamma(\alpha-1)}{\beta^{\alpha-1}}$.\\ 
We now consider the change of variables $u=\frac{x}{x+y}, z=x+y$, which gives $x=uz$, $y=(1-u) z$. We thus have
\bean
\left\langle h_{n+1}^\times,h_{m+1}^\times \right\rangle & = & \frac{(m+n)!}{m!n!} \int_0^\infty \int_0^1 \rho(uz) \rho((1-u)z) (uz)^n ((1-u)z)^m \frac{1}{z^{m+n+1}} z\,dz\,du \\
&=& \frac{(m+n)!}{m!n!} \int_0^1 u^n (1-u)^m \int_0^\infty \rho(uz) \rho((1-u)z) \,dz\,du\,,
\eean 
as desired.
\end{proof}

\section{Sequence defined by induction and a remarkable Gram matrix} \label{ex2}

In this section, we consider a sequence of functions $g_k$ defined by the induction
\bea \label{eq:recupoly}
g_{k+1}(x) = - x g'_k(x) - r_k g_k(x),
\eea
where $(r_k)$ is a real sequence and $g_0\in C^\infty(\R^+)$ such that for all $k\ge0$ and some $\alpha >0$, \bean
\lim_{x\to 0} x^{k-\alpha}g_0^{(k)}(x)= \lim_{x\to \infty} x^{k+1+\alpha}g_0^{(k)}(x)=0.
\eean

\subsection{Condition on $g_0$ ensuring gNB}
\begin{theorem}
The Mellin transform $\widehat{g^{\times}_k}$ is well defined and we have: 
\begin{eqnarray} \label{gtimeshat}
\widehat{g^{\times}_k}(s) & = & -\prod_{j=0}^k(s-r_j)\frac{\zeta(s)}{s}\widehat{g_0}(s), \quad 0<\sg<1.
\end{eqnarray}
If $\d \widehat{g_0}\left(\frac{1}{2}+it\right)\ll e^{-\delta |t|}$, $\delta>0$, then $(g_k)_{k\ge0}$ verifies gNB.
\end{theorem}

\begin{proof}
One can show by induction on $k$ that there exists a family of numbers $(a_{l,k})_{k \ge 0, 0 \le l \le k}$, with $a_{k,k} = (-1)^k$ such that for all $k \ge 0$,
\bea
g_k(x) = \sum_{l=0}^k a_{l,k} x^l g_0^{(l)}(x). 
\eea
Let $s$ with $0<\sg<1$. Due to the assumption on $g_0$, $g_k$ verifies Assumption (M) and $\widehat{g_{k}}$ is well defined for all $k\ge0$. By integration by parts
\begin{eqnarray*}
-\int_0^{\infty} x^{s-1}xg_k'(x)dx
&=&-\int_0^{\infty} x^{s}g_k'(x)dx\\
&=&-\left[x^sg_k(x)\right]_0^\infty+\int_0^{\infty} sx^{s-1}g_k(x)dx\\
&=& s\widehat{g_k}(s),
\end{eqnarray*}
since $\lim_{x\to 0,\infty} x^{\sg}g_k(x)=0$, again by the assumption on $g_0$.

Hence $\widehat{g_{k+1}}(s)  =  (s-r_k)\widehat{g_k}(s)$, and then 
\begin{eqnarray*}
\widehat{g_{k+1}}(s) & = & \prod_{j=0}^k(s-r_j)\widehat{g_0}(s).
\end{eqnarray*}
We found again the polynomial structure of $\widehat{g^\times_k}$, as in the proof of Theorem \ref{gNB1}.
We can then follow these lines to conclude.
\end{proof}

\subsection{Gram matrix}
\begin{prop}
For $k \ge 0$ and $j \ge 1$, we have
\bean
\left\langle g_{k}^\times,g_{j}^\times \right\rangle + \left\langle g_{k+1}^\times,g_{j-1}^\times \right\rangle & = & (1-r_k-r_{j-1}) \left\langle g_{k}^\times,g_{j-1}^\times \right\rangle.
\eean
\end{prop}

\begin{proof}
We compute the scalar products as:
\bean
\int_0^\infty \int_0^{\infty} \left\{\frac{x}{t}\right\}g_k(x)\frac{dx}{x} \int_0^{\infty} \left\{\frac{y}{t}\right\}g_j(y)\frac{dy}{y}\ dt & = & 
\int_0^\infty \int_0^\infty \left\{x\right\}\left\{y\right\} \int_0^{\infty} g_k(tx)g_j(ty)dt \ \frac{dx}{x}\frac{dy}{y}.
\eean
Moreover, using $g_{j}(t)=-tg'_{j-1}(t)-r_{j-1} g_{j-1}(t)$, we compute
\bean
I_{k,j}=I_{k,j}(x,y) & = & \int_0^{\infty} g_k(tx)g_j(ty)dt \\
    & = & \int_0^{\infty} g_k(tx)(-tyg'_{j-1}(ty)-r_{j-1} g_{j-1}(ty))dt \\
    & = & -\int_0^{\infty} g_k(tx)tyg'_{j-1}(ty)dt - r_{j-1}I_{k,j-1}.
\eean   
By integrating by parts, one has
\bean
\int_0^{\infty} g_k(tx)tyg'_{j-1}(ty)dt & = & \left[tg_k(tx)g_{j-1}(ty)\right]_0^\infty - \int_0^{\infty} (txg'_k(tx)+g_k(tx))g_{j-1}(ty)dt \\
     & = & \int_0^{\infty} (g_{k+1}(tx)+(r_k-1)g_k(tx))g_{j-1}(ty)dt\\
     & = & I_{k+1,j-1} +(r_k-1) I_{k,j-1}.
\eean   
Hence,
\bean   
I_{k,j}    & = & -I_{k+1,j-1} -(r_k-1) I_{k,j-1} -r_{j-1} I_{k,j-1} \\
    & = & (1-r_k-r_{j-1}) I_{k,j-1}-I_{k+1,j-1},
\eean
and the result follows due to the linearity of the double integral.
\end{proof}

We now set $G_{k,j}=\left\langle g_{k}^\times,g_{j}^\times \right\rangle$. Choosing $r_q=1/2$ for all $q\ge0$ yields in the previous proposition to a remarkable structure, i.e.
$G_{k,j}+G_{k+1,j-1}=0$. In particular, setting $j=k+1$, we obtain $G_{k,k+1}=0$. This is not a surprise due to the following computation by means of Mellin transform and (\ref{gtimeshat}). We set $s=1/2+it$ and use Mellin-Plancherel isometry:
\bean
G_{k,k+1} & = & \int_{-\infty}^{\infty} \widehat{g_{k}^\times}(s) \overline{\widehat{g_{k+1}^\times}(s)} dt \\
    & = & \int_{-\infty}^{\infty} (s-1/2)^k\overline{(s-1/2)^{k+1}}\left|\frac{\zeta(s)}{s}\widehat{g_0}(s)\right|^2 dt \\
    & = & (-1)^k i\int_{-\infty}^{\infty} t^{2k+1}\left|\frac{\zeta(s)}{s}\widehat{g_0}(s)\right|^2 dt =0.
\eean
The Gram matrix $G$, sort of "alternate" Hankel matrix, is then only determined through its diagonal entries:
\bean
G_{k,k} = \left\langle g_{k}^\times,g_{k}^\times \right\rangle
     =\int_{-\infty}^{\infty} t^{2k} d\nu(t)\,,
\eean
where $\displaystyle d\nu(t) = \left|\frac{\zeta(s)}{s}\widehat{g_0}(s)\right|^2 dt\,.$ The Gram matrix then reads:
\bea
G = \left(
\begin{array}{ccccc}
G_{11}  & 0       & -G_{22} & 0       & \cdots \\
0       & G_{22}  & 0       & -G_{33} & \cdots \\
-G_{22} & 0       & G_{33}  & 0       & \cdots \\
0       & -G_{33} & 0       & G_{44}  & \cdots \\
\vdots  & \vdots  & \vdots  & \vdots  & \ddots 
\end{array}
\right)\,.
\eea
Renumbering the rows and the columns, making appear first the odd indices and then the even ones, leads to the equivalent matrix
\bea
\widetilde{G} = \left(
\begin{array}{cc}
\widetilde{G}_1 & 0 \\ 
0 & \widetilde{G}_2
\end{array} \right)
\eea
where the two blocks $\widetilde{G}_1$ and $\widetilde{G}_2$ are Hankel : 

\bea
\widetilde{G}_1 = \left(
\begin{array}{cccc}
G_{11} & -G_{22} & G_{33} & \cdots  \\ 
-G_{22} & G_{33} & \cdots &  \\
G_{33} & \vdots & & \\
\vdots & & & 
\end{array}\right) \ , \ 
\widetilde{G}_2 = \left(
\begin{array}{cccc}
G_{22} & -G_{33} & G_{44} & \cdots  \\ 
-G_{33} & G_{44} & \cdots &  \\
G_{44} & \vdots & & \\
\vdots & & & 
\end{array}\right).
\eea 


Notice that if one wants to evaluate the squared distance (\ref{Dn2}) by computing the determinant of $G$, we can obtain the determinant of a moment matrix 
$H_{k,j}=\int_{-\infty}^{\infty} t^{k+j}d\nu(t)$ by multiplying the rows of $\widetilde{G}$ by $(-1)^k$ and the columns by $(-1)^j$. 

The study of such determinant falls into the theory of Hankel determinants with symbols that possess power-like singularities, a specific case of more general Fisher-Hartwig singularities. This theory is well established with a finite number of power-like singularities, see e.g. \cite{Kra07,Cha19}. However, the infinite number of zeros of $\zeta(1/2+it)$, which appear in our symbol, seems to be a challenging issue for this theory to apply.

It is also possible to provide simpler expressions for the scalar products $b_k=\langle\chi,g_k^\times\rangle$. Indeed:
\bean
b_k=\int_0^\infty \chi(t) \int_0^{\infty} \left\{\frac{x}{t}\right\}g_k(x)\frac{dx}{x}\ dt = 
\int_0^\infty  \left\{x\right\}\int_0^{\infty}\chi(t)g_k(tx)dt \frac{dx}{x}.
\eean
But
\bean
\int_0^{1}g_{k+1}(tx)dt & = & - \int_0^{1}txg'_{k}(tx)dt -\frac{1}{2} \int_0^{1}g_{k}(tx)dt \\
    & = & -[tg_k(tx)]_0^1 + \int_0^{1}g_{k}(tx)dt -\frac{1}{2} \int_0^{1}g_{k}(tx)dt\\
    & = & -g_k(x) +\frac{1}{2} \int_0^{1}g_{k}(tx)dt.
\eean
Therefore
\bean
b_{k+1} & = & \frac{1}{2}b_k
-\int_0^\infty \left\{x\right\}g_k(x) \frac{dx}{x}.
\eean

\subsection{Examples for $g_0$ and comments}

\subsubsection{The $\Xi$-function}
As an answer to a suggestion by P. Biane and C. Delaunay, it is possible to use the induction (\ref{eq:recupoly}) to find a $g_0$ that produces within $\nu$ the $\Xi$-function:
\bean
\Xi(t)=\xi(s) = \frac{1}{2}s(s-1)\pi^{-s/2}\Gamma\left(\frac{s}{2}\right)\zeta(s), \quad s=1/2+it.
\eean
Indeed, to construct $g_0$ such that  $\widehat{g_0}(s)=\frac{1}{2}(s-1)s^2\Gamma\left(\frac{s}{2}\right)$, we define first $h_0(t)=e^{-t^2}$, so that 
\bean
\widehat{h_0}(s)=\int_0^\infty t^{s-1}e^{-t^2}dt=\int_0^\infty u^{s/2-1/2}\frac{e^{-u}}{2\sqrt{u}}du=\frac{1}{2}\Gamma\left(\frac{s}{2}\right).
\eean
Then, in order to have $\widehat{g_0}(s) = (s^3 - s^2) \widehat{h_0}(s)$, we compute
\bean
g_0(t) &=& \left((-t\frac{d}{dt})^3 - (-t\frac{d}{dt})^2\right) h_0(t)\\
&=& (8t^6-28t^4+12t^2)e^{-t^2}\,.
\eean

Hence, again in the case $r_k=1/2$, we can obtain, taking $\pi^{-1/4}g_0$, 
\bean
G_{k,k} & = & \int_{-\infty}^{\infty} t^{2k}\Xi(t)^2 dt.
\eean
As an historical nod, notice that quantities as $\int_{-\infty}^{\infty} t^{2k}\Xi(t) dt$ or related to, have been used by P\'olya and Hardy to study the zeros of $\zeta$ on the critical line, see e.g. \cite[10.2-10.4, p.256-260]{Tit86}.

\subsubsection{Seed with compact support}

If the seed $g_0$ has a compact support, say $(0,M)$, then the $g_k$'s are also supported on $(0,M)$. This removes the control condition (\ref{Cg}) on the coefficient $c_{k,n}$ as soon as $M_n\ge M$. If one wants to prove RH, one then only needs a density result.

Amazingly, in this compact support case, we then lose the density of the polynomials in the whole space $L^2(\nu)$. Indeed, Ingham \cite{Ing34} remarked (originally for Fourier transform) that if $g_0$ has compact support $[a,b]\subset (0,\infty)$, we cannot have $\widehat{g_0}(s)\ll e^{-\delta |t|}$, $\delta>0$, where again $s=1/2+it$. We transfered here the result for Fourier transform to Mellin transform noticing that 
$$\widehat{g_0}(s)=\int_{-\infty}^\infty e^{us}g_0(e^u)du=\int_{-\infty}^\infty e^{u/2}g_0(e^u)e^{iut}du$$ and that $u\mapsto e^{u/2}g_0(e^u)$ has compact support $[\log(a),\log(b)]$. More precisely, for any decreasing function $\e(t)=o(1)$, there exists a compactly supported function $g_0$ such that $\widehat{g_0}(s)\ll e^{-\e(|t|)|t|}$ if and only if $\d \int_1^\infty \frac{\e(t)}{t}\,dt<\infty$. See \cite[Annexe]{Swa19} for a nice account on these results. But this is precisely incompatible with the condition that ensures the density of the polynomials in weighted $L^2(\R)$-spaces, namely $\d \int_{-\infty}^\infty \frac{\log w(t)}{1+t^2}dt=\infty$. The link between the weight $w$ and the function $\e$ is $\log w(t)=\e(t)t$ here. See \cite[4.8.3 p.77]{Nik12} for many aspects regarding such theorems.  

Let us stress that Mellin isometry involves an integration on the whole real line. On the half line, a density result is obtained by Mergelyan \cite{Mer58} with the condition $\d \int_1^\infty \frac{\log w(t)}{t^{3/2}}dt=\infty$. In our framework, Borichev \cite{Bor20} proved that for all $\e>0$ there exist $Q\in\C[X]$ such that
\bean
\int_{0}^{\infty}\left|\frac{1}{s\varphi(s)}-Q(s)\right|^2 |\varphi(s)|^2dt & < & \epsilon,
\eean
where $\d \varphi(s) = \frac{\zeta(s)}{s} \widehat{g_0}(s)$, $s=1/2+it$, $g_0$ has compact support and verifies $\widehat{g_0}(s)\ll e^{-|t|/\log^2|t|}$.

Although we do not have the density on the whole line $(-\infty,+\infty)$, there exist results regarding the closure of the polynomials, see e.g. \cite{Bor01}. So, we ask for the following question: 

What is the closure of the space generated by the polynomials in $L^2\left(\R,|\varphi(s)|^2dt\right)$ when $g_0$ has compact support?

\bigskip

\section*{Acknowledgement} The authors warmly thank the unknown referee for his/her very careful reading of their manuscript, corrections and suggestions that improved its presentation. They are also very grateful to Vincent Alouin for his remark (\ref{alouin}). This was an important start for our study. The second author thanks Michel Balazard, Sergey Berezin, Philippe Biane, Alexander Borichev, Christophe Delaunay, Sophie Grivaux, Igor Krasovsky, Pierre Lazag, Herv\'e Queff\'elec and Olivier Ramar\'e for helpful references and stimulating discussions.

\end{document}